\documentclass[12pt]{article}
\usepackage{amsmath,amssymb,amsthm,url,times,authblk}

\newtheorem{theorem}{Theorem}
\newtheorem{lemma}{Lemma}
\newtheorem{definition}{Definition}
\newtheorem{remark}{Remark}
\newtheorem{example}{Example}
\allowdisplaybreaks 

\title{Zeros of quasi-orthogonal $q$-Laguerre polynomials}
\author{Pinaki Prasad Kar$^{a}$, Priyabrat Gochhayat$^{b}$}
\affil{Department of  Mathematics, Sambalpur University, Sambalpur, 768019, Odisha, India
\newline $^{a}$\textit{E-mail}: \textup{pinakipk@gmail.com}, $^{b}$\textit{E-mail}: \textup{pgochhayat@gmail.com}}

\begin{document}
\maketitle

\begin{abstract}
We investigate the interlacing of zeros of polynomials
of different degrees
within the sequences of $q$-Laguerre polynomials
$\left\{\tilde{L}_n^{(\delta)}(z;q)\right\}_{n=0}^{\infty}$ characterized by $\delta\in(-2,-1).$
The interlacing of zeros of quasi-orthogonal polynomials $\tilde{L}_n^{(\delta)}(z;q)$
with those of the orthogonal polynomials
$\tilde{L}_m^{(\delta+t)}(z;q), m,n\in\mathbb{N}, t\in\{1,2\}$ is also considered.
New bounds for the least zero of
the (order $1$) quasi-orthogonal $q$-Laguerre polynomials are derived.
\end{abstract}


\vspace{.2in} \noindent {\it {Keywords:}}  Zeros, Interlacing, Stieltjes interlacing,
Quasi-orthogonal polynomials, $q$-Laguerre polynomials.

\section{Introduction}\label{Section-1}
\noindent Let $\{t_n(z)\}_{n=0}^{\infty},$ with  deg($t_n)=n,n\in\mathbb{N},$ be a sequence of orthogonal polynomials
and $w(z)$ be the corresponding weight function on a support $(c,d).$
Then a well known fact about the zeros of $t_n(z)$ is that,
they are real, simple and lie in $(c,d).$
Moreover, the zeros of $t_n(z)$ and $t_{n-1}(z)$ interlace,
if there is exactly one zero of $t_{n-1}(z)$ in between any two consecutive zeros of $t_n(z).$
Similarly, the zeros of the two equal degree polynomials $t_n(z)$ and $u_n(z)$ are interlacing if
there is exactly one zero of $u_n(z)$ (equivalently, $t_n(z)$)
in between any two consecutive zeros of $t_n(z)$ (equivalently, $u_n(z)$).
Another interlacing property given by Stieltjes (cf. \cite[Theorem 3.3.3]{Szego-Book}),
usually called Stieltjes interlacing, states that,
for $m,n\in\mathbb{N}$ with $m<n-1,$
if there exists $m$ open intervals with endpoints at the consecutive zeros of $t_n(z)$
then there is exactly one zero of $u_m(z)$ in each of the $m$ open intervals.
It follows that, if there is at least a zero of $u_m(z)$ lies outside the interval
$(t_{1,n},t_{n,n});$ where $t_{1,n}$ and $t_{n,n},$ respectively,
are the smallest and the largest zero of $t_n(z),$
the Stieltjes interlacing will not retain between
the zeros of $u_m(z)$ and $t_n(z)$ with $n-1>m.$

The departure of zeros of polynomials from the interval of orthogonality
leads to the following definition of quasi-orthogonality:
\begin{definition}\label{def-quasi}\normalfont
Suppose $\{t_n(z)\}_{n=0}^\infty$ is a sequence of polynomials and deg($t_n)=n,n\in\mathbb{N},$
and $w(z)$ is the corresponding weight function on a support $[c,d].$
Then $\{t_n(z)\}_{n=0}^\infty$ is called as quasi-orthogonal sequence of order $r<n$ if
\begin{equation}\label{s1e1}
\int_c^d z^i t_n(z) w(z) dz
\begin{cases}
=0; & i=0,1,\cdots,n-r-1,\\
\neq 0; & i=n-r.
\end{cases}
\end{equation}
\end{definition}
If $t_n(z)$ is a quasi-orthogonal polynomial of order $r$ with respect to $w(z)>0$ on $(c,d),$
then there are minimum $(n-r)$ real and distinct zeros of $t_n(z)$ in $(c,d).$
In \eqref{s1e1}, letting $r=0$ we find that $\{t_n(z)\}_{n=0}^\infty$ is an orthogonal polynomial sequence.
Riesz \cite{Riesz} was the first mathematician who studied quasi-orthogonality of order $1$
while order $2$ was derived by Fej\'er \cite{Fejer}.
Following which, Shohat \cite{Shohat}, Chihara \cite{Chihara},
Dickinson \cite{Dickinson}, Draux \cite{Draux}, etc. considered various general cases.
For more information on quasi-orthogonality, we refer the works by
Brezinski et al. \cite{Brezinski2004,Brezinski2019}, Joulak \cite{Joulak} and the references therein.
Recent contribution on the interlacing of zeros of different quasi-orthogonal polynomials
found in \cite{Driver2017,Driver2018,Driver2016,DriverM2016,Driver2015,Alta2013}.
More recently, in the context of $q$-calculus,
the quasi-orthogonality of classical orthogonal polynomials on $q$-linear and $q$-quadratic lattices
is discussed by Tcheutia et al. \cite{Tcheutia2018}
while Kar et al. in \cite{Kar2020} considered the polynomials that are not present in the
$q$-Askey scheme of hypergeometric polynomials to study their quasi-orthogonality.
Pertinently, very less known on quasi-orthogonality of $q$-polynomials.

In this work, we study the interlacing properties of zeros of $q$-Laguerre polynomials.
Therefore, it is worthy to mention the orthogonality and quasi-orthogonality of these polynomials.

The (monic) $q$-Laguerre polynomials of degree $n$ is defined as
\begin{equation*}
\tilde{L}_n^{(\delta)}(z;q)
=\displaystyle{(-1)^n (q^{\delta+1};q)_n\over q^{n (\delta+n)}}\
_1\phi_1\left(\begin{matrix}q^{-n}\\ q^{\delta+1}\end{matrix};q,-q^{\delta+n+1}z\right), \qquad \delta>-1.
\end{equation*}
These are orthogonal with respect to the weight $\displaystyle{z^{\delta}\over (-z;q)_\infty}$ on $(0,\infty).$
For any value of $\delta,$ the sequence of $q$-Laguerre polynomials is defined by the three term recurrence relation
\begin{equation}\label{s1e3}
\tilde{L}_n^{(\delta)}(z;q)=(z-a_n)\tilde{L}_{n-1}^{(\delta)}(z;q)
-{(1-q^{n-1})(1-q^{\delta+n-1}) \over q^{2\delta+4n-5}} \tilde{L}_{n-2}^{(\delta)}(z;q);\ n=2,3,\cdots,
\end{equation}
where
\begin{equation}\label{s1e4}
a_n = {1-q^{n}+q(1-q^{\delta+n-1})\over q^{\delta+2n-1}},
\end{equation}
and $\tilde{L}_0^{(\delta)}(z;q)\equiv 1,$ $\tilde{L}_1^{(\delta)}(z;q)=z-\displaystyle{1-q^{\delta+1}\over q^{\delta+1}}.$

Tcheutia et al. in \cite[Theorem 2.17]{Tcheutia2018},
proved that for fixed $\delta\in\mathbb{R}$ with $\delta\in(-k-1,-k),$
$\left\{\tilde{L}_n^{(\delta)}(z;q)\right\}_{n=0}^{\infty}$
is quasi-orthogonal of order $k$ with respect to
$\displaystyle{z^{\delta+k}\over (-z;q)_{\infty}}.$
There are minimum $n-k$ real and distinct zeros in $(0,\infty).$
Throughout this paper we denote
$\{z_{i,n}\}_{i=1}^{n}, \{y_{i,n}\}_{i=1}^{n}$ and $\{x_{i,n}\}_{i=1}^{n},$ respectively,
the zeros of $\tilde{L}_n^{(\delta)}(z;q),$ $\tilde{L}_n^{(\delta+1)}(z;q)$
and $\tilde{L}_n^{(\delta+2)}(z;q)$ for fixed $\delta\in(-2,-1)$ and $0<q<1.$

The following result obtained by substituting $\delta$ by $\delta+1$ in \cite[Theorem 2.18]{Tcheutia2018}
is used in our work.
\begin{lemma}\label{s1l1}
The zeros of $\tilde{L}_{n}^{(\delta)}(z;q)$
and $\tilde{L}_{n}^{(\delta+1)}(z;q)$ are interlacing as
\begin{equation}
z_{1,n}<0<y_{1,n}<z_{2,n}<y_{2,n}<\cdots<z_{n,n}<y_{n,n}.\label{s1e5}
\end{equation}
\end{lemma}
\begin{remark}\normalfont\label{s1r1}
Replacing $n$ by $n+1$ in \cite[Theorem 2.18]{Tcheutia2018}, we also obtain
\begin{equation}
z_{1,n+1}<0<y_{1,n}<z_{2,n+1}<y_{2,n}<\cdots<z_{n,n+1}<y_{n,n}<z_{n+1,n+1}.\label{s1e6}
\end{equation}
\end{remark}
In the next section, we show that the zeros of the $q$-Laguerre polynomials
of consecutive degree are not interlacing.
The interlacing results between the zeros of the quasi-orthogonal
and the orthogonal $q$-Laguerre polynomials of the same,
consecutive and non-consecutive degree with different parameter values are also discussed
in section \ref{Section-2}.
In section \ref{Section-3}, we examine the common zeros of polynomials that are not co-prime.
The last section is devoted to the derivation of the inner and outer bound for the
only negative zero of the quasi-orthogonal $q$-Laguerre polynomial.
\section{Zeros of $\boldsymbol{\tilde{L}_{n}^{(\delta)}(z;q)}$ and $\boldsymbol{\tilde{L}_{m}^{(\delta+t)}(z;q)}$ for $\boldsymbol{m\leq n+1,}$ $\boldsymbol{m,n\in\mathbb{N}}$ and $\boldsymbol{t\in\{0,1,2\}}$}\label{Section-2}
Our first theorem proves that
the interlacing does not take place between
the zeros of $\tilde{L}_{n}^{(\delta)}(z;q)$ and $\tilde{L}_{n+1}^{(\delta)}(z;q).$
\begin{theorem}\label{s2t1}
The zeros of $\tilde{L}_{n}^{(\delta)}(z;q)$ and $\tilde{L}_{n+1}^{(\delta)}(z;q)$
are placed in the following order:
\[z_{1,n}<z_{1,n+1}<0<z_{2,n+1}<z_{2,n}<\cdots<z_{n,n}<z_{n+1,n+1}.\]
\end{theorem}
\begin{proof}
Evaluating the mixed $q$-contiguous relation \cite[(4.14)]{Moak1981}
\begin{equation*}
\tilde{L}_{n+1}^{(\delta)}(z;q)=-{1-q^{\delta+n+1}\over q^{\delta+2n+1}}\tilde{L}_{n}^{(\delta)}(z;q)+z\tilde{L}_{n}^{(\delta+1)}(z;q)
\end{equation*}
at the consecutive zeros $z_{i,n}$ and $z_{i+1,n}, i\in\{1,\cdots,n-1\},$ of $\tilde{L}_{n}^{(\delta)}(z;q),$ we obtain
\begin{equation}
\tilde{L}_{n+1}^{(\delta)}(z_{i,n};q)\tilde{L}_{n+1}^{(\delta)}(z_{i+1,n};q)
=z_{i,n}z_{i+1,n} \tilde{L}_{n}^{(\delta+1)}(z_{i,n};q)\tilde{L}_{n}^{(\delta+1)}(z_{i+1,n};q).\label{s2t1e1}
\end{equation}
Now from \eqref{s1e5}, $z_{1,n}z_{2,n}<0$ and $z_{i,n}z_{i+1,n}>0$ for $i\in\{2,\cdots,n-1\}$
while for $i\in\{1,\cdots,n-1\},$ $\tilde{L}_{n}^{(\delta+1)}(z_{i,n};q)\tilde{L}_{n}^{(\delta+1)}(z_{i+1,n};q)<0.$
We conclude from \eqref{s2t1e1} that $\tilde{L}_{n+1}^{(\delta)}(z_{1,n};q) \tilde{L}_{n+1}^{(\delta)}(z_{2,n};q)>0$
and $\tilde{L}_{n+1}^{(\delta)}(z_{i,n};q) \tilde{L}_{n+1}^{(\delta)}(z_{i+1,n};q)<0$ for each $i\in\{2,\cdots,n-1\}.$
This implies that $\tilde{L}_{n+1}^{(\delta)}(z;q)$ has an even number of zeros in $(z_{1,n},z_{2,n})$
and an odd number of zeros in $(z_{i,n},z_{i+1,n})$ for each $i\in\{2,\cdots,n-1\}.$
Hence, $\tilde{L}_{n+1}^{(\delta)}(z;q)$ has minimum $n-2$ zeros in $(z_{2,n},z_{n,n}),$
plus its least zero $z_{1,n+1}<0$ and from \eqref{s1e5} and \eqref{s1e6},
its greatest zero $z_{n+1,n+1}>y_{n,n}>z_{n,n}.$
Thus, we have located $n$ zeros of $\tilde{L}_{n+1}^{(\delta)}(z;q)$ and still we have to find one more zero.
Since there are even number of zeros of $\tilde{L}_{n+1}^{(\delta)}(z;q)$
in $(z_{1,n},z_{2,n})$ where $z_{1,n}<0<z_{2,n}$ for any $n\in\mathbb{N},$
we must have either no zero or two zeros in this single interval.
Because we have precisely one negative zero of $\tilde{L}_{n+1}^{(\delta)}(z;q)$ for $n\in\mathbb{N},$
the only option is $z_{1,n}<z_{1,n+1}<0<z_{2,n+1}<z_{2,n}.$
This proves the theorem.
\end{proof}
It is clear from Theorem \ref{s2t1} that,
the zeros of the quasi-orthogonal polynomials $\tilde{L}_{n}^{(\delta)}(z;q)$
and $\tilde{L}_{n+1}^{(\delta)}(z;q)$ are interlacing in the interval $(0,\infty).$
Furthermore, the zeros of $z \tilde{L}_{n}^{(\delta)}(z;q)$ and
$\tilde{L}_{n+1}^{(\delta)}(z;q)$ interlace.

The next result shows that,
the Stieltjes interlacing fails to occur between the zeros of
$\tilde{L}_{n}^{(\delta)}(z;q)$ and $\tilde{L}_{n-2}^{(\delta)}(z;q).$
\begin{theorem}\label{s2t2}
For $n\geq 3, n\in\mathbb{N},$ if $\tilde{L}_{n}^{(\delta)}(z;q)$ and $\tilde{L}_{n-2}^{(\delta)}(z;q)$ are co-prime,
then the zeros of $z(z-a_n) \tilde{L}_{n-2}^{(\delta)}(z;q)$
and $\tilde{L}_{n}^{(\delta)}(z;q)$ interlace, where $a_n$ is given in \eqref{s1e4}.
However, the Stieltjes interlacing will not take place between the zeros of
$\tilde{L}_{n}^{(\delta)}(z;q)$ and $\tilde{L}_{n-2}^{(\delta)}(z;q).$
\end{theorem}
\begin{proof}
Assume that $\tilde{L}_{n}^{(\delta)}(z;q)$ and $\tilde{L}_{n-2}^{(\delta)}(z;q)$ are co-prime.
Then from \eqref{s1e3}, $\tilde{L}_{n}^{(\delta)}(a_n;q)\neq 0.$
Evaluate \eqref{s1e3} at the consecutive zeros $z_{i,n}$ and $z_{i+1,n},$ $i\in\{2,\cdots,n-1\},$
of $\tilde{L}_{n}^{(\delta)}(z;q)$ to get
\begin{equation}\label{s2t2e1}
{\tilde{L}_{n-1}^{(\delta)}(z_{i,n};q)\tilde{L}_{n-1}^{(\delta)}(z_{i+1,n};q) \over
\tilde{L}_{n-2}^{(\delta)}(z_{i,n};q)\tilde{L}_{n-2}^{(\delta)}(z_{i+1,n};q)}
= {(1-q^{n-1})^2 (1-q^{\delta+n-1})^2 \over q^{4\delta+8n-10} (z_{i,n}-a_n)(z_{i+1,n}-a_n)}.
\end{equation}
For $n$ substituted by $n-1$ in Theorem \ref{s2t1},
we have the zeros of
$\tilde{L}_{n}^{(\delta)}(z;q)$ and $\tilde{L}_{n-1}^{(\delta)}(z;q)$ interlace in $(0,\infty).$
So that for every $i\in\{2,\cdots,n-1\},$ the numerator of left-hand side of \eqref{s2t2e1} is negative.
Also the right-hand side of \eqref{s2t2e1} is positive unless
$a_n \in (z_{i,n},z_{i+1,n})$ for each $i\in\{2,\cdots,n-1\}.$
Hence we deduce from \eqref{s2t2e1} that, for every $i\in\{2,\cdots,n-1\},$
$\tilde{L}_{n-2}^{(\delta)}(z_{i,n};q)$ and $\tilde{L}_{n-2}^{(\delta)}(z_{i+1,n};q)$
have different sign except may be for one pair
$z_{j,n},z_{j+1,n},$ with $z_{j,n}<a_n<z_{j+1,n}, 2\leq j\leq n-1.$
Since $n-2$ intervals are there with endpoints at the consecutive positive zeros of $\tilde{L}_{n}^{(\delta)}(z;q)$
and $\tilde{L}_{n-2}^{(\delta)}(z;q)$ has exactly $n-3$ simple positive zeros,
the zeros of $\tilde{L}_{n-2}^{(\delta)}(z;q),$ together with the point $a_n,$
interlace with the zeros of $\tilde{L}_{n}^{(\delta)}(z;q)$ in $(0,\infty).$
The given interlacing result holds true because from Theorem \ref{s2t1},
we have $z_{1,n-2}<z_{1,n}<0<z_{2,n}.$
Further, from Theorem \ref{s2t1}, $z_{1,n-2}<z_{1,n-1}<z_{1,n}<0<z_{2,n}<z_{2,n-1}<z_{2,n-2},$
the smallest zero $z_{1,n-2},$ of $\tilde{L}_{n-2}^{(\delta)}(z;q),$ lies outside $(z_{1,n},z_{n,n}),$
which shows that the Stieltjes interlacing does not hold between the zeros of
the quasi-orthogonal $q$-Laguerre polynomials
$\tilde{L}_{n}^{(\delta)}(z;q)$ and $\tilde{L}_{n-2}^{(\delta)}(z;q).$
\end{proof}
In Lemma \ref{s1l1}, we have for fixed $\delta\in(-2,-1),$
the zeros of $\tilde{L}_{n}^{(\delta)}(z;q)$ and $\tilde{L}_{n}^{(\delta+1)}(z;q)$ are interlacing.
Upon replacing $n$ by $n-1$ in Remark \ref{s1r1},
we see that the zeros of $\tilde{L}_{n}^{(\delta)}(z;q)$
and $\tilde{L}_{n-1}^{(\delta+1)}(z;q)$ interlace.
The interlacing properties between the zeros of
$\tilde{L}_{n}^{(\delta)}(z;q)$ and $\tilde{L}_{n-2}^{(\delta+1)}(z;q)$ are discussed in the next theorem.
Note that, for fixed $\delta\in(-2,-1)$ and $k\geq 1,$
the sequence $\left\{\tilde{L}_{n}^{(\delta+k)}(z;q)\right\}_{n=0}^{\infty}$
is orthogonal in $(0,\infty)$ because $\delta+k>-1.$
\begin{theorem}\label{s2t3}
For $n\geq 3, n\in\mathbb{N},$
the zeros of $\tilde{L}_{n-2}^{(\delta+1)}(z;q),$
together with the point $\displaystyle{1-q^{\delta+n}\over q^{\delta+2n-1}}=b_n$ (say),
interlace with those of $\tilde{L}_{n}^{(\delta)}(z;q)$ in $(0,\infty)$
provided these two polynomials do not have any common zeros.
\end{theorem}
\begin{proof}
Using \cite[(4.14) and (4.16)]{Moak1981} and \cite[(10)]{Kar2020}, we get
\begin{equation}\label{s2t3e1}
\tilde{L}_{n}^{(\delta)}(z;q) = \left(z-b_n\right) \tilde{L}_{n-1}^{(\delta)}(z;q)
-\left(1-q^{n-1}\over q^{\delta+2n-2}\right) z \tilde{L}_{n-2}^{(\delta+1)}(z;q).
\end{equation}
From \eqref{s2t3e1}, $z_{i,n}\neq b_n$ for any $i\in\{1,\cdots,n\},$
if not, this would contradict our supposition.
Evaluating \eqref{s2t3e1} at the consecutive positive zeros $z_{i,n}$ and $z_{i+1,n},$
$i\in\{2,\cdots,n-1\},$ of $\tilde{L}_{n}^{(\delta)}(z;q)$ yields
\begin{equation}\label{s2t3e2}
{\tilde{L}_{n-1}^{(\delta)}(z_{i,n};q)\tilde{L}_{n-1}^{(\delta)}(z_{i+1,n};q) \over
\tilde{L}_{n-2}^{(\delta+1)}(z_{i,n};q)\tilde{L}_{n-2}^{(\delta+1)}(z_{i+1,n};q)}
= {(1-q^{n-1})^2 z_{i,n} z_{i+1,n} \over q^{2\delta+4n-4} (z_{i,n}-b_n)(z_{i+1,n}-b_n)}.
\end{equation}
Clearly, \eqref{s1e5} shows that, both $z_{i,n}$ and $z_{i+1,n}$ are positive for each $i\in\{2,\cdots, n-1\}.$
Thus, the right-hand side of \eqref{s2t3e2} is positive if and only if
for each $i\in\{2,\cdots, n-1\},$ $b_n\notin (z_{i,n},z_{i+1,n}).$
Also from Theorem \ref{s2t1} we have for every $n\geq 2, n\in\mathbb{N}$ and $i\in\{2,\cdots,n-1\},$
$\tilde{L}_{n-1}^{(\delta)}(z_{i,n};q) \tilde{L}_{n-1}^{(\delta)}(z_{i+1,n};q)<0.$
If $b_n\notin (z_{i,n},z_{i+1,n}),$ for every $i\in\{2,\cdots,n-1\},$
then we deduce from \eqref{s2t3e2} that
the zeros of $\tilde{L}_{n-2}^{(\delta+1)}(z;q)$
and $\tilde{L}_{n}^{(\delta)}(z;q)$ interlace in $(0,\infty).$
Additionally, since the point $b_n$ is not in $(z_{2,n},z_{n,n}),$
the zeros of $(z-b_n) \tilde{L}_{n-2}^{(\delta+1)}(z;q)$ and
the positive zeros of $\tilde{L}_{n}^{(\delta)}(z;q)$ interlace.
On the other hand, if $b_n\in(z_{i,n},z_{i+1,n}),$ then we have $b_n\in (z_{j,n},z_{j+1,n})$ for one $j$ with $2\leq j \leq n-1.$
This implies, in $(z_{j,n},z_{j+1,n}),$ $\tilde{L}_{n-2}^{(\delta+1)}(z;q)$ has no sign change
but it changes sign in each of the remaining $n-3$ intervals with endpoints
at the consecutive positive zeros of $\tilde{L}_{n}^{(\delta)}(z;q).$
Finally, for $i=1,$ \eqref{s2t3e2} becomes
\begin{equation}\label{s2t3e3}
{\tilde{L}_{n-1}^{(\delta)}(z_{1,n};q)\tilde{L}_{n-1}^{(\delta)}(z_{2,n};q) \over
\tilde{L}_{n-2}^{(\delta+1)}(z_{1,n};q)\tilde{L}_{n-2}^{(\delta+1)}(z_{2,n};q)}
= {(1-q^{n-1})^2 z_{1,n} z_{2,n} \over q^{2\delta+4n-4} (z_{1,n}-b_n)(z_{2,n}-b_n)}.
\end{equation}
Since $b_n\in(z_{j,n},z_{j+1,n})$ for one $j,$ $j\in\{2,\cdots,n-1\},$ $b_n\notin(z_{1,n},z_{2,n}).$
We know from Theorem \ref{s2t1} that $z_{1,n-1}<z_{1,n}<0<z_{2,n}<z_{2,n-1}.$
This implies,
$\tilde{L}_{n-1}^{(\delta)}(z_{1,n};q)\tilde{L}_{n-1}^{(\delta)}(z_{2,n};q)>0.$
Hence from \eqref{s2t3e3}, $\tilde{L}_{n-2}^{(\delta+1)}(z_{1,n};q)\tilde{L}_{n-2}^{(\delta+1)}(z_{2,n};q)<0.$
Consequently, $\tilde{L}_{n-2}^{(\delta+1)}(z;q)$ has one positive zero less than $z_{2,n}.$
Thus, in each of the cases, we have the zeros of $(z-b_n) \tilde{L}_{n-2}^{(\delta+1)}(z;q)$
interlace with the positive zeros of $\tilde{L}_{n}^{(\delta)}(z;q),$
if these two polynomials do not have any common zeros.
\end{proof}
For the rest of the results in this section
we study the interlacing results for the zeros of
$\tilde{L}_{n}^{(\delta)}(z;q)$ and $\tilde{L}_{n-k}^{(\delta+2)}(z;q),$
where $\delta\in(-2,-1)$ and $k\in\{0,1,2\}.$
We now show a necessary and sufficient condition for
the zeros of the same degree polynomials $\tilde{L}_{n}^{(\delta)}(z;q)$
and $\tilde{L}_{n}^{(\delta+2)}(z;q)$ to be interlacing.
\begin{theorem}\label{s3t1}
For $n\geq 2, n\in\mathbb{N},$
the zeros of the same degree $q$-Laguerre polynomials
$\tilde{L}_{n}^{(\delta)}(z;q)$ and $\tilde{L}_{n}^{(\delta+2)}(z;q)$ interlace
if and only if the smallest positive zero of $\tilde{L}_{n}^{(\delta)}(z;q),$
that is, $z_{2,n}>\displaystyle{-{1-q^{\delta+1}\over q^{\delta+n+1}}}=c_n$ (say).
\end{theorem}
\begin{proof}
The sequence
$\left\{\tilde{L}_{n}^{(\delta)}(z;q)\right\}_{n=0}^{\infty}$ satisfies
the mixed $q$-contiguous relation (obtained using \cite[(4.12) and (4.15)]{Moak1981})
\begin{equation}\label{s3t1e1}
z \tilde{L}_{n}^{(\delta+2)}(z;q) = \left(z-c_n\right) \tilde{L}_{n}^{(\delta+1)}(z;q)
- {1-q^{\delta+n+1}\over q^{\delta+2n+1}} \tilde{L}_{n}^{(\delta)}(z;q).
\end{equation}
We have from \eqref{s1e5} that $\tilde{L}_{n}^{(\delta)}(z;q)$ and $\tilde{L}_{n}^{(\delta+1)}(z;q)$
are co-prime for each fixed $\delta\in(-2,-1).$
Hence, it follows from \eqref{s3t1e1} that $\tilde{L}_{n}^{(\delta)}(z;q)$ and $\tilde{L}_{n}^{(\delta+2)}(z;q)$
can have at most one common zero at $z=c_n.$
If $c_n<z_{2,n},$ then since the zeros of $\tilde{L}_{n}^{(\delta+2)}(z;q)$ lie in $(0,\infty)$ and $z_{2,n}>0,$
the polynomials $\tilde{L}_{n}^{(\delta)}(z;q)$ and $\tilde{L}_{n}^{(\delta+2)}(z;q)$ are co-prime.
Evaluating \eqref{s3t1e1} at the consecutive zeros $z_{i,n}$ and $z_{i+1,n}, i\in\{1,\cdots,n-1\},$
of $\tilde{L}_{n}^{(\delta)}(z;q),$ gives
\begin{equation}\label{s3t1e2}
{\tilde{L}_{n}^{(\delta+1)}(z_{i,n};q)\tilde{L}_{n}^{(\delta+1)}(z_{i+1,n};q)\over \tilde{L}_{n}^{(\delta+2)}(z_{i,n};q)\tilde{L}_{n}^{(\delta+2)}(z_{i+1,n};q)}
= \displaystyle{z_{i,n}z_{i+1,n}\over (z_{i,n}-c_n)(z_{i+1,n}-c_n)}.
\end{equation}
From \eqref{s1e5}, $z_{1,n}z_{2,n}<0$ while for each $i\in\{2,\cdots,n-1\},$ $z_{i,n}z_{i+1,n}>0.$
Further, for every $i\in\{1,\cdots,n-1\},$ $\tilde{L}_{n}^{(\delta+1)}(z_{i,n};q)$ $\tilde{L}_{n}^{(\delta+1)}(z_{i+1,n};q)<0.$
Since $z_{1,n}<0<c_n<z_{2,n},$
we conclude from \eqref{s3t1e2} that $\tilde{L}_{n}^{(\delta+2)}(z;q)$
changes sign in $(z_{i,n},z_{i+1,n})$ for each $i\in\{1,\cdots,n-1\}.$
Consequently, $\tilde{L}_{n}^{(\delta+2)}(z;q)$ has an odd number of zeros
in $(z_{i,n},z_{i+1,n})$ for any $i\in\{1,\cdots,n-1\}.$
We note from \cite[Theorem 3]{Moak1981} that,
\begin{equation}\label{s3t1e3}
0<y_{1,n}<x_{1,n}<y_{2,n}<x_{2,n}<\cdots<y_{n,n}<x_{n,n}.
\end{equation}
Now from \eqref{s1e5} and \eqref{s3t1e3}, we have $z_{n,n}<x_{n,n},$
where $x_{n,n}$ is the largest zero of $\tilde{L}_{n}^{(\delta+2)}(z;q).$
Therefore, the zeros of $\tilde{L}_{n}^{(\delta)}(z;q)$ and $\tilde{L}_{n}^{(\delta+2)}(z;q)$ are interlacing.

If $c_n>z_{2,n},$ then since we have $z_{1,n}<0<z_{2,n}<c_n$ so from \eqref{s3t1e2},
we deduce that $\tilde{L}_{n}^{(\delta+2)}(z_{1,n};q)$ and
$\tilde{L}_{n}^{(\delta+2)}(z_{2,n};q)$ have same sign.
Thus, $\tilde{L}_{n}^{(\delta+2)}(z;q)$ has an even number of zeros in $(z_{1,n},z_{2,n})$
which confirms that the zeros of the polynomials
$\tilde{L}_{n}^{(\delta)}(z;q)$ and $\tilde{L}_{n}^{(\delta+2)}(z;q)$ are not interlacing.
This completes the proof of the theorem.
\end{proof}
In the next theorem we derive an interlacing result
between the zeros of the above considered polynomials involving the point $c_n.$
\begin{theorem}\label{s3t2}
For $n\geq 2, n\in\mathbb{N},$ let $\tilde{L}_{n}^{(\delta)}(z;q)$
and $\tilde{L}_{n}^{(\delta+2)}(z;q)$ be co-prime
and $z_{2,n}<c_n,$ where $c_n$ is defined in Theorem \ref{s3t1}.
Then the zeros of $\tilde{L}_{n}^{(\delta)}(z;q)$
and $\tilde{L}_{n}^{(\delta+2)}(z;q)$ interlace,
if $c_n\notin(x_{i,n},x_{i+1,n})$ for each $i\in\{1,\cdots,n-1\}.$
Otherwise, the zeros of $\tilde{L}_{n}^{(\delta)}(z;q),$ together with the point $c_n,$
interlace with those of $z \tilde{L}_{n}^{(\delta+2)}(z;q).$
\end{theorem}
\begin{proof}
Evaluating \eqref{s3t1e1} at the consecutive zeros $x_{i,n}$ and $x_{i+1,n}, i\in\{1,\cdots,n-1\},$
of $\tilde{L}_{n}^{(\delta+2)}(z;q),$ we obtain
\begin{multline}
(x_{i,n}-c_n) (x_{i+1,n}-c_n) \tilde{L}_{n}^{(\delta+1)}(x_{i,n};q) \tilde{L}_{n}^{(\delta+1)}(x_{i+1,n};q)\\
= \left({1-q^{\delta+n+1}\over q^{\delta+2n+1}}\right)^2 \tilde{L}_{n}^{(\delta)}(x_{i,n};q) \tilde{L}_{n}^{(\delta)}(x_{i+1,n};q).\label{s3t2e1}
\end{multline}
It is clear from \eqref{s3t1e3} that,
$\tilde{L}_{n}^{(\delta+1)}(x_{i,n};q)\tilde{L}_{n}^{(\delta+1)}(x_{i+1,n};q)<0$
for every $i\in\{1,\cdots,n-1\}.$
If for each $i\in\{1,\cdots,n-1\},$ $c_n\notin(x_{i,n},x_{i+1,n}),$
then the left-hand side of \eqref{s3t2e1} is negative, whence
$\tilde{L}_{n}^{(\delta)}(x_{i,n};q)\tilde{L}_{n}^{(\delta)}(x_{i+1,n};q)<0.$
Since $\tilde{L}_{n}^{(\delta+2)}(z;q)$ has positive zeros
and the least zero of $\tilde{L}_{n}^{(\delta)}(z;q)$ is negative,
the only possible arrangement of zeros is
\[z_{1,n}<0<x_{1,n}<z_{2,n}<x_{2,n}<\cdots<z_{n,n}<x_{n,n}<c_n,\]
which shows the first part of the theorem.

On the other hand, if $c_n\in(x_{j,n},x_{j+1,n})$ for one $j,$ $j\in\{1,\cdots,n-1\},$
then we deduce from \eqref{s3t2e1} that,
$\tilde{L}_{n}^{(\delta)}(x_{j,n};q)\tilde{L}_{n}^{(\delta)}(x_{j+1,n};q)>0$
for one $j, j\in\{1,\cdots,n-1\}.$
This shows that $\tilde{L}_{n}^{(\delta)}(z;q)$
has $n-2$ sign changes between the consecutive zeros of $\tilde{L}_{n}^{(\delta+2)}(z;q)$ in $(0,\infty)$
except for the one interval $(x_{j,n},x_{j+1,n})$ containing the point $c_n.$
Since $\tilde{L}_{n}^{(\delta)}(z;q)$ has $n-1$ positive zeros
and $n-2$ sign changes in $(0,\infty),$ no zero of $\tilde{L}_{n}^{(\delta)}(z;q)$
lies in $(x_{j,n},x_{j+1,n})$ that contains $c_n$ and one zero of
$\tilde{L}_{n}^{(\delta)}(z;q)$ is either $<x_{1,n}$ or $>x_{n,n}.$
From \eqref{s1e5} and \eqref{s3t1e3}, we have $z_{n,n}<y_{n,n}<x_{n,n}.$
Hence, the only option is that the least positive zero $z_{2,n}$ of $\tilde{L}_{n}^{(\delta)}(z;q)$ is $<x_{1,n}.$
Since exactly one zero of $\tilde{L}_{n}^{(\delta)}(z;q)$ is negative,
therefore, the zeros of $(z-c_n) \tilde{L}_{n}^{(\delta)}(z;q)$
and $z \tilde{L}_{n}^{(\delta+2)}(z;q)$ interlace.
This proves the theorem.
\end{proof}
We analyze the interlacing properties between the zeros of
$\tilde{L}_{n}^{(\delta)}(z;q)$ and $\tilde{L}_{n-1}^{(\delta+2)}(z;q)$
in the next result.
\begin{theorem}\label{s3t3}
For $n\geq 2, n\in\mathbb{N},$
the zeros of $z \tilde{L}_{n-1}^{(\delta+2)}(z;q)$
and $\tilde{L}_{n}^{(\delta)}(z;q)$ interlace.
\end{theorem}
\begin{proof}
Using \cite[(4.12) and (4.14)]{Moak1981}, we obtain
\begin{equation}\label{s3t3e1}
\left({1-q^{\delta+n+1} \over 1-q^n}\right) \tilde{L}_{n}^{(\delta)}(z;q)
= {q^n (1-q^{\delta+1})\over (1-q^{n})} \tilde{L}_{n}^{(\delta+1)}(z;q) + z \tilde{L}_{n-1}^{(\delta+2)}(z;q).
\end{equation}
Evaluating \eqref{s3t3e1} at the consecutive zeros $z_{i,n}$ and $z_{i+1,n}, i\in\{1,\cdots,n-1\},$
of $\tilde{L}_{n}^{(\delta)}(z;q)$ yields
\begin{equation}\label{s3t3e2}
{q^{2n} (1-q^{\delta+1})^2 \over (1-q^n)^2} \tilde{L}_{n}^{(\delta+1)}(z_{i,n};q) \tilde{L}_{n}^{(\delta+1)}(z_{i+1,n};q)
= z_{i,n} z_{i+1,n} \tilde{L}_{n-1}^{(\delta+2)}(z_{i,n};q) \tilde{L}_{n-1}^{(\delta+2)}(z_{i+1,n};q).
\end{equation}
For any $i\in\{1,\cdots,n-1\},$ \eqref{s1e5} implies
$\tilde{L}_{n}^{(\delta+1)}(z_{i,n};q) \tilde{L}_{n}^{(\delta+1)}(z_{i+1,n};q)<0.$
Also for $i=1,$ $z_{i,n} z_{i+1,n}<0$ while $z_{i,n} z_{i+1,n}>0$ for $i=2,\cdots,n-1.$
Therefore, from \eqref{s3t3e2}, we deduce that $\tilde{L}_{n-1}^{(\delta+2)}(z;q)$
does not change sign in $(z_{1,n},z_{2,n})$ and
changes sign in $(z_{i,n},z_{i+1,n})$ for every $i\in\{2,\cdots,n-1\}.$
Thus, we have minimum $n-2$ zeros of $\tilde{L}_{n-1}^{(\delta+2)}(z;q)$ in $(z_{2,n},z_{n,n})$
and in $(z_{1,n},z_{2,n}),$ there is either no zero or two zeros of $\tilde{L}_{n-1}^{(\delta+2)}(z;q).$
Since $\tilde{L}_{n-1}^{(\delta+2)}(z;q)$ has exactly $n-1$ zeros, it has no zero in $(z_{1,n},z_{2,n}).$
Hence, the remaining one zero must be larger than $z_{n,n}.$
Since $z_{1,n}<0<z_{2,n},$ the given interlacing property holds true.
This proves the theorem.
\end{proof}
\begin{remark}\label{s2r1}\normalfont
From \eqref{s1e5}, $\tilde{L}_{n}^{(\delta)}(z;q)$ and $\tilde{L}_{n}^{(\delta+1)}(z;q)$ are co-prime.
This implies, $\tilde{L}_{n}^{(\delta)}(z;q)$ and $\tilde{L}_{n-1}^{(\delta+2)}(z;q)$ cannot be co-prime.
Otherwise, any common zero of these polynomials is also a zero of $\tilde{L}_{n}^{(\delta+1)}(z;q).$
\end{remark}
Next theorem shows the interlacing between the zeros of
$\tilde{L}_{n}^{(\delta)}(z;q)$ and $\tilde{L}_{n-2}^{(\delta+2)}(z;q).$
\begin{theorem}\label{s3t4}
For $n\geq 3, n\in\mathbb{N},$
the zeros of $\tilde{L}_{n}^{(\delta)}(z;q)$ and $\tilde{L}_{n-2}^{(\delta+2)}(z;q)$ interlace in $(0,\infty).$
Moreover, the zeros of $z(z+c_{n-1})\tilde{L}_{n-2}^{(\delta+2)}(z;q)$
and $\tilde{L}_{n}^{(\delta)}(z;q)$ interlace,
where $c_n$ is given in Theorem \ref{s3t1}.
\end{theorem}

\begin{proof}
Using \eqref{s2t3e1} and \cite[(4.14) and (4.15)]{Moak1981}, we get
\begin{equation}\label{s3t4e1}
\left({1-q^{n-1}\over 1-q^{\delta+n}}\right) z^2 \tilde{L}_{n-2}^{(\delta+2)}(z;q)
= {q^{\delta+2n}\over 1-q^{\delta+n}} c_{n} \tilde{L}_{n}^{(\delta)}(z;q) + (z+c_{n-1}) \tilde{L}_{n-1}^{(\delta)}(z;q).
\end{equation}
Since $\tilde{L}_{n}^{(\delta)}(z;q)$ and $\tilde{L}_{n-1}^{(\delta)}(z;q)$ are co-prime in $(0,\infty),$
from \eqref{s3t4e1}, the only possible common zero of $\tilde{L}_{n}^{(\delta)}(z;q)$
and $\tilde{L}_{n-2}^{(\delta+2)}(z;q)$ in $(0,\infty)$ is $-c_{n-1}.$
But for $\delta\in(-2,-1)$ and $0<q<1,$ $-c_{n-1}<0$
and all the zeros of $\tilde{L}_{n-2}^{(\delta+2)}(z;q)$ are positive.
This implies, the polynomials $\tilde{L}_{n}^{(\delta)}(z;q)$ and $\tilde{L}_{n-2}^{(\delta+2)}(z;q)$
are co-prime and hence, $z_{i,n}\neq -c_{n-1}$ for any $i\in\{1,\cdots,n\}.$
Evaluating \eqref{s3t4e1} at the consecutive zeros $z_{i,n}$ and $z_{i+1,n}, i\in\{2,\cdots,n-1\},$
of $\tilde{L}_{n}^{(\delta)}(z;q),$ we have
\begin{equation}\label{s3t4e2}
{\tilde{L}_{n-2}^{(\delta+2)}(z_{i,n};q)\tilde{L}_{n-2}^{(\delta+2)}(z_{i+1,n};q) \over
\tilde{L}_{n-1}^{(\delta)}(z_{i,n};q)\tilde{L}_{n-1}^{(\delta)}(z_{i+1,n};q)}
= {(1-q^{\delta+n})^2 (z_{i,n}+c_{n-1}) (z_{i+1,n}+c_{n-1}) \over (1-q^{n-1})^2 z_{i,n}^2 z_{i+1,n}^2}.
\end{equation}
We know from Theorem \ref{s2t1} that,
$\tilde{L}_{n-1}^{(\delta)}(z_{i,n};q) \tilde{L}_{n-1}^{(\delta)}(z_{i+1,n};q)<0$ for $i\in\{2,\cdots,n-1\}.$
For $\delta\in(-2,-1),$ $-c_{n-1}<0,$ so $-c_{n-1}<z_{2,n}$ and hence for any $i\in\{2,\cdots,n-1\},$
the right-hand side of \eqref{s3t4e2} is positive.
Therefore, we deduce from \eqref{s3t4e2} that, for each $i\in\{2,\cdots,n-1\},$
$\tilde{L}_{n-2}^{(\delta+2)}(z_{i,n};q)$ and
$\tilde{L}_{n-2}^{(\delta+2)}(z_{i+1,n};q)$ have different sign.
This means, the zeros of $\tilde{L}_{n-2}^{(\delta+2)}(z;q)$
and $\tilde{L}_{n}^{(\delta)}(z;q)$ interlace in $(0,\infty).$

Now evaluating \eqref{s3t4e1} at the zeros $z_{1,n}$ and $z_{2,n},$ of $\tilde{L}_{n}^{(\delta)}(z;q),$ we obtain
\begin{equation}\label{s3t4e3}
{\tilde{L}_{n-2}^{(\delta+2)}(z_{1,n};q)\tilde{L}_{n-2}^{(\delta+2)}(z_{2,n};q)\over
\tilde{L}_{n-1}^{(\delta)}(z_{1,n};q)\tilde{L}_{n-1}^{(\delta)}(z_{2,n};q)}
= {(1-q^{\delta+n})^2 (z_{1,n}+c_{n-1})(z_{2,n}+c_{n-1})\over
(1-q^{n-1})^2 z_{1,n}^2 z_{2,n}^2}.
\end{equation}
Since the zeros of $\tilde{L}_{n-2}^{(\delta+2)}(z;q)$ lie in $(z_{2,n},z_{n,n}),$
no zero of $\tilde{L}_{n-2}^{(\delta+2)}(z;q)$ are in $(z_{1,n},z_{2,n})$
which means it does not change sign in $(z_{1,n},z_{2,n}).$
From Theorem \ref{s2t1},
$\tilde{L}_{n-1}^{(\delta)}(z_{1,n};q) \tilde{L}_{n-1}^{(\delta)}(z_{2,n};q)$ $>0.$
Therefore, from \eqref{s3t4e3} we conclude that $-c_{n-1}\notin(z_{1,n},z_{2,n})$
and since $-c_{n-1}<0,$ it follows that $-c_{n-1}<z_{1,n}.$
Thus, the only possibility of configuration of zeros is as
\begin{equation*}
-c_{n-1}<z_{1,n}<0<z_{2,n}<x_{1,n-2}<\cdots<x_{n-2,n-2}<z_{n,n}.
\end{equation*}
This follows the result.
\end{proof}
\section{Common zeros of $\boldsymbol{\tilde{L}_{n}^{(\delta)}(z;q)}$}\label{Section-3}
In this section, we deal with the $q$-Laguerre polynomials that are not co-prime.
In particular, we discuss the interlacing properties between the zeros of polynomials considered
in Theorem \ref{s2t2} and Theorem \ref{s3t2}.

\begin{theorem}\label{nst1}
  For $n\geq 3, n\in\mathbb{N},$ if $\tilde{L}_{n}^{(\delta)}(z;q)$ and $\tilde{L}_{n-2}^{(\delta)}(z;q)$ are not co-prime,
  then $a_n$ is a common zero of these polynomials and the zeros of $\tilde{L}_{n-2}^{(\delta)}(z;q)$
  interlace with the non-common zeros of $\tilde{L}_{n}^{(\delta)}(z;q)$ in $(0,\infty),$
  where $a_n$ is given in \eqref{s1e4}.
\end{theorem}
\begin{proof}
If $\tilde{L}_{n}^{(\delta)}(z;q)$ and $\tilde{L}_{n-2}^{(\delta)}(z;q)$ have common zeros in $(0,\infty),$
then from \eqref{s1e3}, $a_n$ is the only common zero of these polynomials
since $\tilde{L}_{n}^{(\delta)}(z;q)$ and $\tilde{L}_{n-1}^{(\delta)}(z;q)$ are co-prime in $(0,\infty)$ by Theorem \eqref{s2t1}.
Suppose $z_{1,n-2}^{\prime}<z_{2,n-2}^{\prime}<\cdots<z_{n-2,n-2}^{\prime}$
are the $n-2$ non-common zeros of $\tilde{L}_{n}^{(\delta)}(z;q)$ in $(0,\infty).$
Then $a_n$ belongs to one of the intervals $\left(z_{i,n-2}^{\prime},z_{i+1,n-2}^{\prime}\right),$ $i\in\{1,\cdots,n-3\}.$
Evaluating \eqref{s1e3} at the consecutive non-common zeros $z_{i,n-2}^{\prime}$ and $z_{i+1,n-2}^{\prime},$ $i\in\{1,\cdots,n-3\},$
of $\tilde{L}_{n}^{(\delta)}(z;q),$ such that $a_n\notin \left(z_{i,n-2}^{\prime},z_{i+1,n-2}^{\prime}\right),$ we obtain
\begin{multline}\label{nst1e1}
  {\left(1-q^{n-1}\right)^2 \left(1-q^{\delta+n-1}\right)^2 \over q^{4\delta+8n-10}}
  \tilde{L}_{n-2}^{(\delta)}(z_{i,n-2}^{\prime};q) \tilde{L}_{n-2}^{(\delta)}(z_{i+1,n-2}^{\prime};q) \\
   = \left(z_{i,n-2}^{\prime}-a_n\right) \left(z_{i+1,n-2}^{\prime}-a_n\right)
   \tilde{L}_{n-1}^{(\delta)}(z_{i,n-2}^{\prime};q) \tilde{L}_{n-1}^{(\delta)}(z_{i+1,n-2}^{\prime};q).
\end{multline}
Since the positive zeros of $\tilde{L}_{n}^{(\delta)}(z;q)$ and $\tilde{L}_{n-1}^{(\delta)}(z;q)$ are interlacing,
we deduce from \eqref{nst1e1} that $\tilde{L}_{n-2}^{(\delta)}(z;q)$ has an odd number of zeros
in each of the intervals $\left(z_{i,n-2}^{\prime},z_{i+1,n-2}^{\prime}\right),$ $i\in\{1,\cdots,n-3\},$ that does not contain $a_n.$
Since $\tilde{L}_{n-2}^{(\delta)}(z;q)$ has $n-4$ non-common zeros in $(0,\infty),$
there are at most $n-4$ such intervals.
Therefore, $a_n=z_{j,n}$ where $j\in\{3,\cdots,n-1\}.$
Hence, the $n-3$ zeros of $\tilde{L}_{n-2}^{(\delta)}(z;q)$ (including $a_n$)
interlace with the $n-2$ non-common zeros of $\tilde{L}_{n}^{(\delta)}(z;q)$ in $(0,\infty).$
This proves the theorem.
\end{proof}

\begin{theorem}\label{nst2}
  Suppose $\tilde{L}_n^{(\delta)}(z;q)$ and $\tilde{L}_{n}^{(\delta+2)}(z;q)$ are not co-prime
  and the smallest positive zero of $\tilde{L}_{n}^{(\delta)}(z;q),$ that is, $z_{2,n}<c_n$ for $n\geq 2, n\in\mathbb{N},$
  where $c_n$ is given in Theorem \ref{s3t1}.
  Then the above polynomials have a common zero at $z=c_n$
  and the positive zeros of $\tilde{L}_{n}^{(\delta)}(z;q)$ interlace with the non-common zeros of $\tilde{L}_{n}^{(\delta+2)}(z;q).$
\end{theorem}
\begin{proof}
We have by \eqref{s1e5}, $\tilde{L}_{n}^{(\delta)}(z;q)$ and $\tilde{L}_{n}^{(\delta+1)}(z;q)$ are co-prime.
Therefore, from \eqref{s3t1e1}, $c_n$ is the only common zero of
$\tilde{L}_{n}^{(\delta)}(z;q)$ and $\tilde{L}_{n}^{(\delta+2)}(z;q)$ in $(0,\infty).$
If $x_{1,n-1}^{\prime}<x_{2,n-1}^{\prime}<\cdots<x_{n-1,n-1}^{\prime}$
are the $n-1$ non-common zeros of $\tilde{L}_{n}^{(\delta+2)}(z;q),$
then $c_n$ lies in at most one interval $\left(x_{i,n-1}^{\prime},x_{i+1,n-1}^{\prime}\right),$ $i\in\{1,\cdots,n-2\}.$
Evaluating \eqref{s3t1e1} at the consecutive non-common zeros $x_{i,n-1}^{\prime}$ and $x_{i+1,n-1}^{\prime},$
$i\in\{1,\cdots,n-2\},$ of $\tilde{L}_{n}^{(\delta+2)}(z;q)$ such that $c_n\notin\left(x_{i,n-1}^{\prime},x_{i+1,n-1}^{\prime}\right),$ yields
\begin{multline}
\left(x_{i,n-1}^{\prime}-c_n\right) \left(x_{i+1,n-1}^{\prime}-c_n\right)
\tilde{L}_{n}^{(\delta+1)} \left(x_{i,n-1}^{\prime};q\right) \tilde{L}_{n}^{(\delta+1)} \left(x_{i+1,n-1}^{\prime};q\right)\\
= \left({1-q^{\delta+n+1}\over q^{\delta+2n+1}}\right)^2
  \tilde{L}_{n}^{(\delta)} \left(x_{i,n-1}^{\prime};q\right) \tilde{L}_{n}^{(\delta)} \left(x_{i+1,n-1}^{\prime};q\right). \label{nst2e1}
\end{multline}
We have from \eqref{s3t1e3} that,
the zeros of $\tilde{L}_{n}^{(\delta+1)}(z;q)$ and $\tilde{L}_{n}^{(\delta+2)} (z;q)$ interlace.
Thus, we deduce from \eqref{nst2e1} that
$\tilde{L}_{n}^{(\delta)}(z;q)$ has an odd number of zeros in
$\left(x_{i,n-1}^{\prime},x_{i+1,n-1}^{\prime}\right)$ for each $i\in\{1,\cdots,n-2\},$
that does not contain the point $c_n.$
From Theorem \ref{s3t2}, we have $z_{2,n}<x_{1,n}$ when $c_n\in\left(x_{1,n},x_{n,n}\right).$
Since $\tilde{L}_{n}^{(\delta)}(z;q)$ has $n-3$ non-common zeros in $(x_{1,n},x_{n,n}),$
there are at most $n-3$ such intervals.
Hence, $c_n=x_{j,n}$ where $j\in\{2,\cdots,n-1\}.$
Thus, the $n-2$ non-common positive zeros of $\tilde{L}_{n}^{(\delta)}(z;q),$ together with the point $c_n,$
interlace with the $n-1$ non-common zeros of $\tilde{L}_{n}^{(\delta+2)}(z;q).$
This completes the proof.
\end{proof}
We illustrate Theorem \ref{nst1} and Theorem \ref{nst2} by the following numeric examples using {\it Mathematica 12} :
\begin{example}\label{ex1}\normalfont
If we take $q=0.997, n=26$ and $\delta=-1.121695,$ the polynomials $\tilde{L}_{24}^{(-1.121695)}(z;0.997)$
and $\tilde{L}_{26}^{(-1.121695)}(z;0.997)$ have a common zero in $(0,\infty),$ namely,
$a_{26}=0.167473.$
Further, the zeros of $\tilde{L}_{24}^{(-1.121695)}(z;0.997)$ interlace with
the non-common zeros of $\tilde{L}_{26}^{(-1.121695)}(z;0.997)$ in $(0,\infty).$
\end{example}
\begin{example}\label{ex2}\normalfont
When $q=0.94, n=26$ and $\delta=-1.92598,$ then $c_{26}=0.278236$ is a common zero of
$\tilde{L}_{26}^{(-1.92598)}(z;0.94)$ and $\tilde{L}_{26}^{(-1.92598+2)}(z;0.94).$
Moreover, the positive zeros of $\tilde{L}_{26}^{(-1.92598)}(z;0.94)$
interlace with the non-common zeros of $\tilde{L}_{26}^{(-1.92598+2)}(z;0.94).$
\end{example}

\section{Bounds for the smallest zero of $\boldsymbol{\tilde{L}_{n}^{(\delta)}(z;q)}$}\label{Section-4}
We derive the inner and outer bounds for the negative zero of $\tilde{L}_{n}^{(\delta)}(z;q).$
To find the bounds, we make use of the following contiguous relations:
\begin{equation}\label{s4e1}
z^3 \tilde{L}_{n-2}^{(\delta+3)}(z;q) = {(1-q^{\delta+n+1})\over (1-q^{n-1})} b_n
\left(z-A_n\right) \tilde{L}_{n-1}^{(\delta)}(z;q)
+ \left(z-A_{n-\delta-2}\right) \tilde{L}_{n}^{(\delta)}(z;q)
\end{equation}
and
\begin{multline}\label{s4e2}
z^4 \tilde{L}_{n-2}^{(\delta+4)}(z;q) =
{A_{n-\delta-2}\ b_n ((1-q^n)+q(1-q^{\delta+n+1})) \over q^{n} (1-q^{\delta+1})} \left(z-B_n\right) \tilde{L}_{n-1}^{(\delta)}(z;q)\\
+ \left(z^2+{(1+q)(1-q^{\delta+2})\over q^{\delta+n+1}}z - {A_{n-\delta-2} (1-q^{\delta+3}) \over q^{\delta+n+1}}\right) \tilde{L}_{n}^{(\delta)}(z;q),
\end{multline}
where
$A_n=\displaystyle{(1-q^{\delta+1})(1-q^{\delta+2})\over q^{\delta+1} (1-q^{\delta+n+1})},$
$B_n=\displaystyle{(1-q^{\delta+1})(1-q^{\delta+3})\over q^{\delta+1} ((1-q^n)+q(1-q^{\delta+n+1}))}$
and $b_n$ is given in Theorem \ref{s2t3}.

\noindent Clearly, \eqref{s4e1} and \eqref{s4e2} are the
application of \cite[(4.14) and (4.15)]{Moak1981} in \eqref{s3t4e1}.

\begin{theorem}\label{s4t1}
For $n\geq 2, n\in\mathbb{N},$
the negative zero $z_{1,n}$ of $\tilde{L}_{n}^{(\delta)}(z;q)$ satisfies
\begin{equation}\label{s4t1e1}
-c_{n-1} < B_n < z_{1,n} < A_n < 0,
\end{equation}
where $A_n, B_n$ and $c_n$ are respectively given in \eqref{s4e1}, \eqref{s4e2} and in Theorem \ref{s3t1}.
\end{theorem}

\begin{proof}
At $z=z_{1,n},$ \eqref{s4e2} gives
\begin{equation}\label{s4t1e2a}
z_{1,n}^4 \tilde{L}_{n-2}^{(\delta+4)}(z_{1,n};q) =
{A_{n-\delta-2}\ b_n ((1-q^n)+q(1-q^{\delta+n+1})) \over q^{n} (1-q^{\delta+1})}
\left(z_{1,n}-B_n\right) \tilde{L}_{n-1}^{(\delta)}(z_{1,n};q).
\end{equation}
For all $n$ and $\delta,$ we have
\begin{equation}\label{s4t1e3}
\lim_{z\to -\infty} \tilde{L}_{n}^{(\delta)}(z;q) = \infty.
\end{equation}
This implies that $(z_{1,n}-B_n)$ must be positive since all other factors in \eqref{s4t1e2a} are positive,
that means
\begin{equation}\label{s4t1e4}
B_n<z_{1,n}.
\end{equation}
Similarly, we have from \eqref{s4e1} at $z=z_{1,n}$ that
\begin{equation}\label{s4t1e4a}
z_{1,n}^3 \tilde{L}_{n-2}^{(\delta+3)}(z_{1,n};q) = {(1-q^{\delta+n+1}) \over (1-q^{n-1})} b_n
\left(z_{1,n}-A_n\right) \tilde{L}_{n-1}^{(\delta)}(z_{1,n};q).
\end{equation}
Using \eqref{s1e5} and \eqref{s4t1e3} we observe that, the factor $(z_{1,n}-A_n)$ must be negative
since the left-hand side of \eqref{s4t1e4a} is negative while
the first and final factor on the right-hand side are positive,
so that
\begin{equation}\label{s4t1e5}
z_{1,n}<A_n.
\end{equation}
Now for $n\geq 2,$ we see that
\[(1-q^n)+q(1-q^{n-2})>0,\]
which implies
\[(1-q^n)+q(1-q^{\delta+n+1})>q^{n-1}(1-q^{\delta+3}).\]
Suitable interpretation of the above inequality leads to
\begin{equation}\label{s4t1e6}
-c_{n-1}<B_n.
\end{equation}
Therefore, from \eqref{s4t1e4}, \eqref{s4t1e5} and \eqref{s4t1e6}, we have the inequality \eqref{s4t1e1}.
\end{proof}
\begin{remark}\label{s4r1}\normalfont
By using the fact as in \cite[(14.21.18)]{Koekoek}, when we take limit as $q\to 1$ in \eqref{s4t1e1},
it recovers the bounds, obtained by Driver and Muldoon in \cite{Driver2015},
for the negative zero of the quasi-orthogonal Laguerre polynomials of order $1.$
\end{remark}
Using {\it Mathematica 12}, we examine the inner and outer bounds for the smallest zero of $\tilde{L}_{n}^{(\delta)}(z;q)$
for different values of the parameters  $q,$ $n$ and $\delta$ such that $\delta\in(-2,-1)$
which are listed in Table \ref{table}.
\begin{table}[h!]
\footnotesize
\centering
\caption{Inner and outer bounds of $z_{1,n}$ for different values of $q, n$ and $\delta.$}
\vspace{1mm}
\begin{tabular}{|l| c| r|r|r|r|}
\hline
$q$ & $n$ & $\delta$   & $B_n$           & $z_{1,n}$                & $A_n$            \\[1ex]
\hline
    &     & $-1.1$     & $-0.11032$      & $-0.110294$              & $-0.10681$       \\[1ex]\cline{3-6}
    & \raisebox{1.5ex}{$2$}
    & $-1.81$          & $-0.505545$     & $-0.448837$              & $-0.20526$       \\[1ex]\cline{2-6}
    &     & $-1.1$     & $-0.104315$     & $-0.104286$              & $-0.100269$      \\[1ex]\cline{3-6}
\raisebox{1.5ex}{$0.23$} & \raisebox{1.5ex}{$7$}
    & $-1.81$          & $-0.46738$      & $-0.40912$               & $-0.169571$      \\[1ex]\cline{2-6}
    &     & $-1.1$     & $-0.104311$     & $-0.104283$              & $-0.100265$      \\[1ex]\cline{3-6}
    & \raisebox{1.5ex}{$12$}
    & $-1.81$          & $-0.467357$     & $-0.409097$              & $-0.169552$      \\[1ex]
\hline
    &     & $-1.1$     & $-0.00598206$   & $-0.00597785$            & $-0.00580805$    \\[1ex]\cline{3-6}
    & \raisebox{1.5ex}{$2$}
    & $-1.81$          & $-0.0360942$    & $-0.0294349$             & $-0.0152326$     \\[1ex]\cline{2-6}
    &     & $-1.1$     & $-0.00219276$   & $-0.00219076$            & $-0.0020879$     \\[1ex]\cline{3-6}
\raisebox{1.5ex}{$0.89$} & \raisebox{1.5ex}{$7$}
    & $-1.81$          & $-0.0114906$    & $-0.00913766$            & $-0.0038382$     \\[1ex]\cline{2-6}
    &     & $-1.1$     & $-0.0016198$    & $-0.00161831$            & $-0.00153785$    \\[1ex]\cline{3-6}
    & \raisebox{1.5ex}{$12$}
    & $-1.81$          & $-0.0083227$    & $-0.00661085$            & $-0.00270734$    \\[1ex]
\hline
\end{tabular}
\label{table}
\end{table}
\newpage
\section*{Acknowledgments}
Present investigation is supported under OURIIP, Govt. of Odisha, India, Sanction Number - 1040/69/OSHEC.

\end{document}